\documentclass[12pt]{article}

\usepackage{amsmath}
\usepackage{amsthm}
\usepackage{amsfonts}
\usepackage{mathrsfs}
\usepackage{stmaryrd}
\usepackage{setspace}
\usepackage{fullpage}
\usepackage{amssymb}
\usepackage{breqn}
\usepackage{enumitem}
\usepackage{bbold}
\usepackage{authblk}
\usepackage{comment}
\usepackage{hyperref}
\usepackage{pgf,tikz}
\usepackage{graphicx}
\usepackage{subcaption}

\newtheorem{thm}{Theorem}[section]

\newtheorem{lemma}[thm]{Lemma}

\newtheorem{prop}[thm]{Proposition}

\newtheorem{clm}[thm]{Claim}
\newtheorem{claim}[thm]{Claim}

\newtheorem*{theorem*}{Theorem}
\newtheorem*{prop*}{Proposition}

\newcommand\ex{\ensuremath{\mathrm{ex}}}
\newcommand\tr{\ensuremath{\mathrm{Tr}}}

\newcommand\cF{{\mathcal F}}

\newcommand\cH{{\mathcal H}}

\newcommand\cN{{\mathcal N}}

\usepackage{lineno}

\newcommand{\vep}{\varepsilon}

\newcommand{\ignore}[1]{}

\title{On forbidding graphs as traces of hypergraphs}
\author[1]{D\'aniel Gerbner\footnote{Alfr\'ed R\'enyi Institute of Mathematics, HUN-REN, E-mail: \texttt{gerbner@renyi.hu.}} \and Michael E. Picollelli\footnote{Department of Mathematics, California State University San Marcos, San Marcos, CA 92096. E-mail: \texttt{mpicollelli@csusm.edu}}}
\date{}

\begin{document}

\maketitle

\begin{abstract}
    We say that a hypergraph $\mathcal{H}$ contains a graph $H$ as a trace if there exists some set $S\subset V(\mathcal{H})$ such that $\mathcal{H}|_S=\{h\cap S: h\in E(\mathcal{H})\}$ contains a subhypergraph isomorphic to $H$. We study the largest number of hyperedges in 3-uniform hypergraphs avoiding some graph $F$ as trace. In particular,
    we improve a bound given by Luo and Spiro in the case $F=C_4$, and obtain exact bounds for large $n$ when $F$ is a book graph.
\end{abstract}

\section{Introduction}

A fundamental theorem in extremal Combinatorics is due to Tur\'an \cite{T} and determines the largest number of edges in $n$-vertex $K_k$-free graphs (the case $k=3$ was proved earlier by Mantel \cite{man}). More generally, given a family $\cF$ of graphs, $\ex(n,\cF)$ denotes the largest number of edges in $n$-vertex graphs that do not contain any member of $\cF$ as a (not necessarily induced) subgraph, and if there is one forbidden subgraph, we use the simpler notation $\ex(n,F)$ instead of $\ex(n,\{F\})$.
The Erd\H os-Stone-Simonovits theorem \cite{ES1966,ES1946} determines the asymptotics of $\ex(n,\cF)$ in the case $\cF$ does not contain any bipartite graphs. The bipartite case is much less understood and is the subject of extensive research, see \cite{fusi} for a survey.

There is a natural analogue of this problem for hypergraphs and was already asked by Tur\'an. Given a family $\cF$ of hypergraphs, we denote be $\ex_r(n,\cF)$ the largest number of edges in an $r$-uniform hypergraph that does not contain any member of $\cF$. This problem is much more complicated, for example we still do not know the asymptotics in the next obvious question, when the complete 4-vertex 3-uniform hypergraph is forbidden. A relatively recent line of research is to consider \textit{graph-based hypergraphs}. This is an informal common name of hypergraph classes that are obtained from graphs by enlarging their edges according to some set of rules. Extremal results concerning such hypergraphs were collected in Section 5.2. in \cite{gp}.

The most studied graph-based hypergraphs are the following. The \textit{expansion} $F^{(r)+}$ of $F$ is obtained by adding $r-2$ new vertices to each edge such that each new vertex is added to only one edge, see \cite{mubver} for a survey on expansions. A \textit{Berge copy} of $F$ is obtained by adding  $r-2$ new vertices to each edge arbitrarily. The new vertices may be already in $F$ or not, and they may be added to any number of hyperedges. More precisely, we say that a hypergraph $\cF$ is a Berge copy of $F$ if there is a bijection $f$ between the edges of $F$ and the hyperedges of $\cF$ such that for each edge $e$ we have $e\subset f(e)$. Observe that there can be several non-isomorphic $r$-uniform Berge copies of $F$, the expansion being one of them. We denote by Berge-$F$ the family of Berge copies of $F$.
Each Berge copy of $F$ is defined by a graph copy of $F$ on a subset of the vertices, we call that the \textit{core} of the Berge-$F$. Berge hypergraphs were defined (generalizing the notion of hypergraph cycles due to Berge \cite{Be87}) by Gerbner and Palmer \cite{gp1}.

Here we study a third type of graph based hypergraphs. We denote by $\tr(F)$ the family of Berge copies of $F$ where the vertices added to the edges of $F$ are each outside $V(F)$. In other words, the \textit{trace} of these Berge copies is $F$ on $V(F)$, i.e., $f(e)\cap V(F)=e$ for each edge $e$ of $F$. These were called \textit{induced Berge} in \cite{fl}.

The maximum number of hyperedges in hypergraphs with some forbidden traces have long been studied. For example, the celebrated Sauer Lemma \cite{sau,she,vc} deals with the case $\cH$ does not contain the power set of a $t$-element set as a trace.

The Tur\'an problem for these graph-based hypergraphs is closely related to the so-called \textit{generalized Tur\'an problems}. Given two graphs $H$ and $G$, we denote by $\cN(H,G)$ the number of copies of $H$ contained in $G$. Given an integer $n$ and graphs $H$ and $F$, we let $\ex(n,H,F)=\max\{\cN(H,G): \text{ $G$ is an $n$-vertex $F$-free graph}\}$. After several sporadic results, the systematic study of this function was initiated by Alon and Shikhelman \cite{ALS2016}.

It is easy to see that if we take the vertex sets of $r$-cliques in an $F$-free graph as hyperedges, the resulting graph is Berge-$F$-free. Therefore, we have $\ex(n,K_r,F)\le \ex_r(n,\text{Berge-}F)\le\ex_r(n,\tr(F))\le\ex_r(n,F^{(r)+})$, where the second and third inequality follows from $F^{(r)+}\in \tr(F)\subset \text{Berge-}F$. Stronger connection was established for these cases: $\ex_r(n,\text{Berge-}F)\le \ex(n,K_r,F)+\ex(n,F)$ \cite{gp2}, $\ex_r(n,\tr(F))=\Theta(\max_{s\le k}\ex(n,K_s,F))$ \cite{fl} and $\ex_r(n,F^{(r)+})=\ex(n,K_r,F)+O(n^{r-1})$ \cite{gerbner}.

The first to study forbidden graphs as traces were Mubayi and Zhao \cite{muzh}. They studied the case $F=K_k$ (in fact they considered the more general case when complete hypergraphs are forbidden as traces). They observed that in the case $r<k$ we have $\ex_r(n,\tr(K_k))=\ex_r(n,F^{(r)+})$ for sufficiently large $n$, and the exact value of that was determined by Pikhurko \cite{pikhu}. In the case $r\ge k$, Mubayi and Zhao conjectured that the extremal construction for sufficiently large $n$ is the following. We take a $K_k$-free graph on $n-r+k-1$ vertices with $\ex(n,K_{k-1},K_k)$ copies of $K_{k-1}$. Note that this is the so-called \textit{Tur\'an graph} by a theorem of Zykov \cite{zykov}. Then we take a set $U$ of $r-k+1$ new vertices, and pick as hyperedges the union of $U$ with the vertex set of any $(r-1)$-clique of the Tur\'an graph.
Mubayi and Zhao proved this conjecture asymptotically if $k=3$, i.e., they showed that $\ex_r(n,\tr(K_3))=n^2/4+o(n^2)$. They proved this exactly if $r=3$. An earlier version of this paper contained a proof of the $k=3$ case of their conjecture for any $r$, i.e., the proof of $ex_r(n,\tr(K_3))=\lfloor(n-r+2)^2/4\rfloor$, if $n$ is sufficiently large. However, this statement follows from an old theorem of Frankl and Pach \cite{frapac} for every $n$, which also implies the extremal examples are constructed by adding a common set of $r-2$ vertices to every edge of maximum complete bipartite graph on $n-r+2$ vertices.

We extend this result to book graphs in the 3-uniform case. The \textit{book graph} $B_t$ consists of $t$ triangles sharing an edge, i.e., $B_t=K_{1,1,t}$.

\begin{thm}\label{book} For every $t$, if $n$ is sufficiently large, then we have
    $\ex_3(n,\tr(B_t))=\lfloor (n-1)^2/4\rfloor$, with equality only for the $3$-uniform hypergraph formed by adding a common vertex to every edge of a maximum complete bipartite graph on $n-1$ vertices.
\end{thm}

We remark that for any non-star $F$ and any $r$ and $n$ we have $\ex_r(n,\tr(F))\ge \ex(n-r+2,F)$, since if we add the same $r-2$ new vertices to each edge of an $F$-free graph, we obtain a $\tr(F)$-free hypergraph.


Other specific graphs $F$ such that $\ex_r(n,\tr(F))$ has been studied include stars \cite{fl,qg} and $K_{2,t}$ \cite{lusp,qg}.
In particular, Luo and Spiro \cite{lusp} showed $ n^{3/2}/2+o(n^{3/2})\le \ex_3(n,\tr(C_4))\le 5 n^{3/2}/6+o(n^{3/2})$.
We improve the constant factor in the main term of the upper bound.


\begin{thm}\label{4cyc}
    $\ex_3(n,\tr(C_4))\le (1+\sqrt{2}) n^{3/2}/4+o(n^{3/2})$.
\end{thm}

Finally, we show a connection of $\ex_r(\tr(n,F))$ and generalized supersaturation. Given a graph $F$ and a positive integer $m$, the supersaturation problem deals with the minimum number of copies of $F$ in $n$-vertex graphs with at least $m$ edges. In the generalized version, we are also given a graph $H$ and the $n$-vertex graphs contain at least $m$ copies of $H$. Such problems were studied in \cite{cnr,gnv,hapa}. Note that this is equivalent to studying the most number of copies of $H$ when we are given an upper bound on the number of copies of $F$.

\begin{prop}\label{propi}
    There is an $n$-vertex graph $G$ with $O(n^{|V(F)|-1})$ copies of $F$ such that $\ex_r(n,\tr(F))\le \cN(K_r,G)$.
\end{prop}

Obviously this connection to generalized Tur\'an problems is less useful than the previously mentioned ones, partly because there are less results on generalized supersaturation problems. In the non-degenerate case when $\chi(F)>r$, a special case of a result of Halfpap and Palmer \cite{hapa} states that $\cN(F,G)=o(n^{|V(F)|})$ implies $\cN(K_r,G)\le (1+o(1))\ex(n,K_r,F)$. Therefore, with Proposition \ref{propi} we obtain the (already known) asymptotics of $\ex(n,\tr(F))$ in this case. We could not find any application of Proposition \ref{propi} that gives new results, although it plays a small role in the proof of Theorem \ref{book}.


\section{Proofs}

The \textit{shadow graph} of a hypergraph $\cH$ has vertex set $V(\cH)$ and $uv$ is an edge if and only if there is a hyperedge in $\cH$ containing both $u$ and $v$. If $\cH$ contains a $\tr(F)$ (or any Berge-$F$), then the core of that is a copy of $F$ in the shadow graph. The converse is not true, for example if $\cH$ is $r$-uniform, then the shadow graph contains copies of $K_r$, even if $\cH$ is Berge-$K_r$-free.

Note that in the shadow graph of a $\tr(F)$-free hypergraph, each copy of $F$ has to contain an edge $uv$ with the property that each hyperedge containing $uv$ contains another vertex from that copy of $F$. In multiple proofs below, we will pick such an edge for each copy of $F$ and call it the \textit{special edge} of that copy of $F$.

We can show the following generalization of Proposition \ref{propi}.

\begin{prop}\label{propi2}
    If $\cH$ is $\tr(F)$-free, then the shadow graph $G$ of $\cH$ contains $O(n^{|V(F)|-1})$ copies of $F$.
\end{prop}

Note that hyperedges of $\cH$ create distinct copies of $K_r$ in $G$, thus $|E(\cH)|\le \cN(K_r,G)$, hence the above proposition implies Proposition \ref{propi}.

\begin{proof}
    Let us fix an edge $uv$ and count the number of copies of $F$ that have $uv$ as special edge. Let $h$ be a hyperedge containing $u$ and $v$. Then each copy of $F$ containing $uv$ contains also at least one of the other $r-2$ vertices of $h$. Therefore, we can count the copies of $F$ containing $uv$ by picking a non-empty subset of the other vertices of $h$ ($O(1)$ ways), then picking the rest of the vertices of $F$ ($O(n^{|V(F)|-3})$ ways) and then picking a copy of $F$ on those $|V(F)|$ vertices ($O(1)$ ways). As there are $O(n^2)$ special edges, the proof is complete.
\end{proof}

Let us continue with the proof of Theorem \ref{4cyc}. We will use the following lemma due to Luo and Spiro (Lemma 3.1 in \cite{lusp}). We say that a subset of a hyperedge is \textit{light} if exactly one hyperedge contains it, and \textit{heavy} otherwise.

\begin{lemma}[Luo, Spiro \cite{lusp}]\label{lemls} Let $\cH$ be a $\tr(C_4)$-free 3-uniform hypergraph. Let $\cH_2$ denote the subhypergraph of $\cH$ consisting of the hyperedges that do not contain light edges. Then every edge of the shadow graph is in at most two hyperedges of $\cH_2$.
\end{lemma}

We restate Theorem \ref{4cyc} here for convenience.

\begin{theorem*}
    $\ex_3(n,\tr(C_4))\le (1+\sqrt{2}) n^{3/2}/4+o(n^{3/2})$.
\end{theorem*}

We denote hyperedges consisting of vertices $u,v,w$ by $uvw$. Note that there are 4-uniform hyperedges in the proof below, but we do not use the analogous notation, because we use that for 4-cycles. We let $uvwx$ denote the 4-cycle with edges $uv,vw,wx,xu$.

\begin{proof}
Let $\cH$ be a 3-uniform $Tr(C_4)$-free hypergraph, and let $G$ denote its shadow graph. For vertices $u,v$ of $\cH$, we denote by $d_3(u,v)$ the number of hyperedges containing both $u$ and $v$, and by $d_2(u,v)$ the number of common neighbors of $u$ and $v$ in $G$.

\begin{claim}\label{refclaim1}  $d_2(u,v)\le d_3(u,v)+3$.
\end{claim}

\begin{proof}[Proof of Claim]
Assume we have four vertices $w_1,w_2,w_3,w_4$ such that for every $i$, $uvw_i$ is not a hyperedge in $\cH$, yet $uw_i$ and $vw_i$ are in $G$. Consider the 4-cycle $uw_1vw_2$ in $G$. Since it is not the core of a $Tr(C_4)$ in $\cH$, there is a hyperedge inside these four vertices. Since no hyperedge contains $u$, $v$ and $w_i$, we have that either $uw_1w_2$ or $vw_1w_2$ is in $\cH$. This holds for each pair $w_iw_j$, which implies that the 4-cycle $w_1w_2w_3w_4$ is the core of a $Tr(C_4)$, where each edge is extended by $u$ or $v$ to a hyperedge.
\end{proof}

Let us return to the proof of the theorem. Consider now the triangles in $G$ such that the vertices do not form a hyperedge of $\cH$. For each edge of $G$, there are at most three such triangles, and each triangle has three edges, thus there are at most $|E(G)|$ such triangles. This implies that the total number of triangles in $G$ is at most $|E(G)|+|E(\cH)|=O(n^{3/2})$.

Consider now two non-adjacent vertices $u$ and $v$. By Claim \ref{refclaim1} we have $d_2(u,v)\le 3$. Observe that if $w_1$ and $w_2$ are in the common neighborhood of $u$ and $v$, then at least one of $uw_1w_2$ and $vw_1w_2$ is in $\cH$ (because of the 4-cycle $uw_1vw_2$). In particular, $w_1$ and $w_2$ are adjacent in $G$.
Assume that $d_2(u,v)=3$ and let $w_1,w_2,w_3$ be the common neighbors of $u$ and $v$. Then these three vertices form a triangle, which we denote by $T(u,v)$.

Assume that $uw_1w_2$, $uw_1w_3$ and $uw_2w_3$ are each in $\cH$. Since $uw_1vw_2$ is not a core of a $Tr(C_4)$, one of the edges incident to $v$, say $vw_1$ is contained only in hyperedges of $\cH$ inside this 4-cycle, which must be $vw_1w_2$ since $uv$ is not in the shadow graph of $\cH$. Consider now the 4-cycle $uw_1vw_3$. By similar reasoning, either $vw_1$ or $vw_3$ is only in hyperedges inside this 4-cycle, thus $vw_1w_3\in\cH$. This contradicts that $vw_1$ is contained only in $vw_1w_2$ in $\cH$. Therefore, $uw_iw_j$ is not in $\cH$ for some $1\le i<j\le 3$.

We obtained that $u$ (and analogously $v$) is the common neighbor of the endpoints of an edge $w_iw_j$ in $G$ such that the triangle $uw_iw_j$ does not form a hyperedge in $\cH$, and $w_iw_j$ is in a triangle in $G$. For each edge of $G$, there are at most 3 such vertices, thus for each triangle in $G$, there are at most 27 such pairs. Therefore, the number of pairs $(u,v)$ such that $uv$ is not an edge in $G$ and $u$ and $v$ have three common neighbors is at most $27\cN(K_3,G)=O(n^{3/2})$. Let $E'$ denote the set of non-adjacent pairs $u,v$ with $d_2(u,v)=3$. We have

$\sum_{uv\in E(G)}d_2(u,v)\le \sum_{uv\in E(G)}d_3(u,v)+3=3|\cH|+3|E(G)|=O(n^{3/2})$, $\sum_{(u,v)\in E'}d_2(u,v)=3|E'|=O(n^{3/2})$ and $\sum_{(u,v)\not\in E(G)\cup E'}d_2(u,v)\le \sum_{(u,v)\not\in E(G)\cup E'}2\le 2\binom{n}{2}$. Therefore,

\[\sum_{u,v\in V(G)}d_2(u,v)=\sum_{uv\in E(G)}d_2(u,v)+\sum_{(u,v)\in E'}d_2(u,v)+\sum_{(u,v)\not\in E(G)\cup E'}d_2(u,v)\le (2+o(1))\binom{n}{2}.\]

This implies that $\sum_{v\in V(G)}\binom{d(v)}{2}\le (2+o(1))\binom{n}{2}$, and then $|E(G)|\le (1+o(1))\frac{\sqrt{2}}{2}n^{3/2}$ follows by Jensen's inequality.

\smallskip 

 Let $\cH_1$ be the subhypergraph consisting of the hyperedges that contain a light edge. Let us pick a light edge from each hyperedge of $\cH_1$ and let $G_1$ denote the resulting graph. Then $G_1$ is $C_4$-free, thus $|E(\cH_1)|=|E(G_1)|\le \ex(n,C_4)=(1+o(1))n^{3/2}/2$.

    Let $\cH_2$ denote the rest of the hyperedges of $\cH$ and $G_2$ denote the shadow graph of $\cH_2$. Then $G_2$ may contain copies of $C_4$, but each copy $uvwx$ of $C_4$ contains an edge $uv$ such that each hyperedge in $\cH$ that contains $u$ and $v$ also contains $w$ or $x$. We say that $uv$ is a \textit{special edge} for this $C_4$. In particular, $u$ and $v$ are contained together in at most two hyperedges of $\cH$, and since $uv\in E(G_2)$, we have that $uvw, uvx\in\cH$ (with at least one of them in $\cH_2$). This also implies that there is a $K_4$ on $u,v,w,x$ in $G$.

    Let us consider a 4-cycle $uvwx$ in $G_2$ with special edge $uv$, and assume that both $uvw$ and $uvx$ are in $\cH_2$. Then the $K_4$ on $u,v,w,x$ is in $G_2$ and the 4-cycle $uwvx$ also has a special edge, without loss of generality $uw$. This shows that $uwx\in\cH$ (since $uw$ is in $G_2$). Then $uwx\in\cH_2$ since each subedge is in $G_2$. Observe that we have found two hyperedges in $\cH_2$ containing $ux$. Now we have two possibilities. Either there are no further hyperedges containing $ux$, or there are some hyperedges containing $ux$ (they must be in $\cH_1$ because of Lemma \ref{lemls}).

    Now we move the hyperedges $uvw$, $uvx$ and $uwx$ to $\cH_1$ and the edges $uv$ and $uw$ to $G_1$. In the first case, if $ux$ is not contained in any hyperedges of $\cH_1$, we also move $ux$ to $G_1$. We repeat this as long as we can. We denote the hypergraph we obtain this way from $\cH_1$ by $\cH_1'$ and the graph we obtain from $G_1$ by $G_1'$.

\begin{clm}
    $G_1'$ is $C_4$-free.
\end{clm}


\begin{proof}[Proof of Claim] First we show that each hyperedge of $\cH$ contains at most one edge of $G_1'$, except for the hyperedges added to $\cH_1'$. This holds for $G_1$ since the edges are light and we picked only one from each hyperedge.
The hyperedges of $\cH_1$ do not contain any of the edges $uv$ in $E(G_1')\setminus E(G_1)$, since the only hyperedges of $\cH$ that contain $uv$ and $uw$ (and $ux$ if that is also added) are $uvw$, $uwx$ and $uvx$, and they are in $\mathcal{H}_2$.




Observe that the special edge $ab$ of a $C_4$ $abcd$ in $G_1'$ would be contained in a hyperedge that contains two edges of $G_1'$, thus in a hyperedge that was added to $\cH_1'$. Since the only hyperedges containing $ab$ are $abc$ and $abd$, we have that $ab$ is moved from $G_2$ to $G_1'$ together with, say $ad$ (and potentially $ac$). This implies that $abc\in \cH_2$. Then $bc\in E(G_2)$, thus $bc$ was also moved to $G_1'$. But that movement cannot be at the same time when we moved $ab$ and $ad$ (and potentially $ac$) to $G_1'$, thus cannot be at the same time when we moved $abc$ to $\cH_1'$. But we move each edge together with the two hyperedges of $\cH_2$ that contain that edge, thus we have to move $ab$ at the same time when we move $abc$, a contradiction.
\end{proof}





Let us return to the proof of the theorem.
Let $f=|E(G_1')\setminus E(G_1)|$.
The number of hyperedges in $\cH_1'$ is at most the number of edges in $G_1$ plus 3/2 times the number of new edges in $G_1'$, i.e., $|E(\cH_1')|\le (1+o(1))n^{3/2}/2+f/2$.
We also have $|E(\cH_1)|=|E(G_1)|\le (1+o(1))n^{3/2}/2-f$.

Let $\cH_2'$ denote the hyperedges of $\cH$ that are not in $\cH_1'$, thus $\cH_2'$ is a subhypergraph of $\cH_2$. Let $G_2'$ denote the graph of we obtain from $G_2$ by deleting the edges we moved to $G_1'$.
We have that $|E(G_2')|= |E(G_2)|-f$. Recall that by Lemma \ref{lemls}, each edge of $G_2$ is in at most two hyperedges of $\cH_2$.

Let $Y$ denote the set of edges in $G_2'$ that are in exactly two hyperedges of $\cH_2$ and $g:=|Y|$. We next provide an upper bound on $g$ in terms of $f$ and edges of $G_2'$ that are not in $Y$.




Let $uv\in Y$, then the two corresponding hyperedges $uvw$ and $uvz$ create a $C_4$ with a chord $uv$ in $G_2$. That $C_4$ has a special edge, say $uw$, then $uwz$ is also in $\cH$ (if there are multiple special edges and one of them was moved to $G_1'$, we pick that as $uw$). If $uwz\not \in \cH_2$, then $uw\in G_2'$ and is contained in less than two hyperedges of $\cH_2$. We say that $uw$ \textit{belongs} to $uv$.

\begin{clm}
    Each edge $uw$ not in $Y$ belongs to at most one edge $uv$.
\end{clm}

\begin{proof}[Proof of Claim] We have that $uw$ is a special edge of a 4-cycle $uwvz$, and $uwz\not\in\cH_2$ implies that $wz$ must be a light edge. Observe that $uw$ can belong only to an edge in the unique hyperedge in $\cH_2$ that contains $u,w$, thus to $uv$ or $wv$. If $uw$ belongs to $wv$ as well, then $uw$ is the special edge of a 4-cycle of the form $uwz'v$ in $G_2$. We have $z'\neq z$ since $wz$ is light. But then $uwz'$ must be in $\cH$ because of $uwzv'$, but cannot be in $\cH$ because $uw$ is only in hyperedges inside $uvwz$.
\end{proof}

Assume now that $uwz\in\cH_2$ and we have not moved $uw$ to $G_1'$. The only reasons can be that we have already moved at least one of the edges of the $K_4$ on $u,v,w,z$ to $G_1'$, or we have already moved at least one of the three hyperedges we would have moved with it to $\cH_1'$, because of another subedge, another $K_4$. Let $x$ be the fourth vertex of this other $K_4$. Note that if we have moved one of those edges, we have also moved all hyperedges of $\cH_2$ containing that edge, in particular we have moved at least one of the hyperedges $uwv$, $uwz$ and $uvz$.

No matter which of the hyperedges $uwv$, $uwz$ and $uvz$ we have moved earlier, it contains two subedges containing $u$. Each such subsedge is contained by two of the hyperedges $uwv$, $uwz$ and $uvz$. Therefore, none of the hyperedges in $\cH_2$ containing such a subedge contains $x$. This implies that at least two 3-sets containing $x$ inside the $K_4$ are not in $\cH_2$. But we only move things to $\cH_1'$ if at least three hyperedges inside the $K_4$ are in $\cH_2$, a contradiction. This implies that we have not moved any of these three hyperedges $uwv$, $uwz$ and $uvz$ earlier. But these hyperedges in $\cH_2$ and whenever we move an edge to $G_1'$, we also move the hyperedges that contain that edge from $\cH_2$ to $\cH_1'$, thus each subedges, each edges inside $u,v,w,z$ have not been moved, therefore, we could move $uw$ to $G_1'$, a contradiction.

Finally, assume that $uwz\in\cH_2$ and we have moved $uw$ to $G_1'$. It was because of the $K_4$ on $u,v,w,z$ since the hyperedges containing $u,w$ are $uwv$ and $uwz$.
Then we also moved at least one adjacent edge to $G_1'$ at the same time. If we moved three edges to $G_1'$, then there are three edges inside the $K_4$ that remain in $G_2'$.
If we moved two edges to $G_1'$, then there are four edges inside the $K_4$ that remain in $G_2'$. This can happen at most $f/2$ times, thus there are at most $2f$ edges in $Y$ where this occurs.

We have obtained for each edge of $G_2'$ in $Y$ a unique edge not in $Y$, with at most $2f$ exceptions. This implies that $g=|Y|\le |E(G_2')|/2+f=|E(G_2)|/2+f/2$.
Then in $G_2$, we have at most $g+f\le |E(G_2)|/2+3f/2$ edges that are contained in exactly two hyperedges of $\cH_2$.
This shows that the total number of edge-hyperedge incidences between $G_2$ and $\cH_2$ is at most $3|E(G_2)|/2+3f/2$, hence $|E(\cH_2)|\le |E(G_2)|/2+f/2$.
Recall that $|E(G_1)|+f/2=|E(G_1')|-f/2\le (1+o(1))n^{3/2}/2-f/2$ and
$|E(G_2)|\le |E(G)|-|E(G_1)|\le (1+o(1))\frac{\sqrt{2}}{2}n^{3/2}-|E(G_1)|$.

Combining the upper bounds we obtained on $|E(\cH_1)|$ and $|E(\cH_2)|$, we obtain that $|E(\cH)|=|E(\cH_1)|+|E(\cH_2)|\le |E(G_1)|+ |E(G_2)|/2+f/2\le \frac{|E(G_1)|}{2}+\frac{|E(G)|}{2}+f/2\le \frac{|E(G_1)|}{2}+(1+o(1))\frac{\sqrt{2}}{4}n^{3/2}+f/2=(1+o(1))n^{3/2}/4+(1+o(1))\frac{\sqrt{2}}{4}n^{3/2}$.
\end{proof}


We finish the paper by proving Theorem \ref{book}. Recall that it states $\ex_3(n,\tr(B_t))=\lfloor (n-1)^2/4\rfloor$ for sufficiently large $n$.

\begin{proof}[Proof of Theorem \ref{book}]
     The lower bound is given by taking a complete bipartite graph on $n-1$ vertices, and adding the same new vertex to each edge.

    Let $\cH$ be a $\tr(B_t)$-free hypergraph and let $\cH_1$ be the subhypergraph consisting of the hyperedges having a light subedge. Let $G_1$ be a graph obtained by taking a light subedge for each hyperedge in $\cH_1$. Let $\cH_2$ denote the rest of the hyperedges of $\cH$ and $G_2$ denote the shadow graph of $\cH_2$. Let $G$ denote the shadow graph of $\cH$.

    \begin{claim}\label{clami0}
        $|E(\cH_2)|=o(n^2)$.
    \end{claim}

    \begin{proof}[Proof of Claim]
        $G_2$ contains $O(n^{t+1})$ copies of $B_t$ by Proposition \ref{propi2}. Hence by the removal lemma \cite{efr} there is a set $A$ of $o(n^2)$ edges in $G_2$ such that each copy of $B_t$ contains an edge from $A$.

    We claim that each edge of $G_2$ is in at most $3t-3$ hyperedges of $\cH_2$. We use a lemma of Luo and Spiro \cite{lusp}, who showed the analogous statement in the case of forbidden $\tr(K_{2,t})$. More precisely, they showed that if an edge $xy$ is in at least $3t-2$ hyperedges of $\cH_2$, then there is a $\tr(K_{2,t})$ in $\cH$ such that $x$ and $y$ are the vertices in the smaller part of the core $K_{2,t}$. To find a $\tr(B_t)$ in $\cH$, we need to find a hyperedge containing $x,y$ such that the third vertex of that hyperedge is not in this core $K_{2,t}$. This is doable if $t>1$, since we can pick any of the $2t-2$ hyperedges containing $x,y$ and avoiding the other $t$ vertices of the core $K_{2,t}$.

    Now we are ready to count the number of triangles in $G_2$. There are at most $(3t-3)|A|=o(n^2)$ triangles containing an edge in $A$. The rest of $G_2$ is $B_t$-free, thus contains $o(n^2)$ triangles by a result of Alon and Shikhelman \cite{ALS2016}. As each hyperedge of $\cH_2$ creates a triangle in $G_2$, we have $|E(\cH_2)|=o(n^2)$, completing the proof.
    \end{proof}

\begin{claim}\label{clmi}
    For any edge $uv$ in $G$, $u$ and $v$ have at most $3t-2$ common neighbors in $G_1$.
\end{claim}

\begin{proof}[Proof of Claim]
    Let us assume that $w_1,\dots,w_{3t-1}$ are each adjacent to both $u$ and $v$ in $G_1$. There is a hyperedge in $\cH$ containing $u$, $v$ and a third vertex, which is not among $w_1,\dots,w_{3t-2}$ without loss of generality. The hyperedge of $\cH_1$ containing $uw_1$ has a third vertex $w_1'$. Then $w_1'$ is not $v$, since then we would not pick $vw_1$ as an edge of $G_1$. If $w_1'=w_i$, we can assume without loss of generality that $i=2$. Similarly, the third vertex of the hyperedge containing $vw_1$ is not $u$ and if it is $w_i$, we can assume that $i\le 3$.

    Then we consider the third vertices of the hyperedges containing $uw_3$ and $vw_3$, analogously we can assume that none of them are $w_i$ with $i>6$. We continue this way, the third vertices of the hyperedges containing $uw_{3j-2}$ are not  $w_i$ with $i>3j$. Note that those vertices are not of the form $w_{3k-2}$ with $k<j$ either, since we have already studied the single hyperedge containing $uw_{3k-2}$ and the single hyperedge containing $uw_{3k-2}$, and the third vertex is not $w_{3j-2}$. Then the $B_t$ with vertices $u$, $v$ and $w_1,w_4, \dots, w_{3t-2}$ is the core of a $Tr(B_t)$, a contradiction completing the proof.
\end{proof}

Let us return to the proof of the theorem. The above claim implies that $G_1$ is $B_{3t-1}$-free (it is easy to see that $G_1$ is actually $B_t$-free). Observe that $|E(\cH_1)|=|E(G_1)|$, hence we are done by Claim \ref{clmi} unless $|E(G_1)|\ge n^2/4-o(n^2)$.
Therefore,
by the Erd\H os-Simonovits stability theorem \cite{erd1,erd2,simi} we can obtain a complete bipartite graph with parts $A$ and $B$,
by adding and deleting $o(n^2)$ edges. We may assume $A$ and $B$ form a maximum cut in $G_1$, and thus each vertex has no more neighbors in its own part than in the other part.


Our aim now is to refine the estimates on both $|E(\cH_2)|$ and $|E(G_1)|$ sufficient to establish the theorem. With that in mind, we let $\zeta,\vep$ be small but fixed positive constants, for which $\zeta$ is chosen to be small relative to $\vep$.

We then partition the vertices of $G$ into the following five sets: let $A'\subset A$ denote the set of vertices in $A$ that are adjacent in $G_1$ to all but at most $\zeta n$ vertices of $B$; we define $B' \subseteq B$ analogously.  We next let $A'' \subseteq A\setminus A'$ consist of the vertices adjacent to at least $n/4 + \zeta n$ vertices of $B'$ in $G_1$, and $B'' \subseteq B\setminus B'$ consist of the vertices adjacent to at least $n/4 + \zeta n$ vertices of $A'$ in $G_1$.  Finally, let $X = V(G)\setminus (A'\cup A'' \cup B' \cup B'')$ consist of all remaining vertices.

Since $|E(G_1)| \ge n^2/4 - o(n^2)$ and since $o(n^2)$ edges are missing between $A$ and $B$ or lie within $A$ or $B$, we have that $|A|-|A'|=o(n)$, $|B|-|B'|=o(n)$, and both $|A'|$ and $|B'|$ lie within $o(n)$ of $n/2$. Consequently, every pair of vertices in $A' \cup A''$ and in $B' \cup B''$ has $\Omega(n)$ common neighbors in $G_1$, hence both are independent sets in $G$ by Claim~\ref{clmi}. Moreover, we may also assume that $|X \cup A'' \cup B''| \le \zeta n$.\\

Let $x = |X|$.

\begin{claim}\label{clmii}
   $|E(G_1)| \le n^2/4 + x(-n/4 + 3\zeta n)$.
\end{claim}

\begin{proof}[Proof of Claim]
Let $u \in X$: trivially $u$ has at most $\zeta n$ neighbors in $X \cup A''\cup B''$.  If $u$ has a neighbor $v \in A'$ then $u$ can be adjacent to the at most $\zeta n$ non-neighbors of $v$ in $B'$ along with at most $3t-2$ common neighbors by Claim~\ref{clmi}. Similarly, if $u$ has a neighbor in $B'$ then it has at most $\zeta n + 3t-2$ neighbors in $A'$.  Consequently, if $u$ has neighbors in both $A'$ and $B'$ then it has less than $3\zeta n$ such neighbors in total, whereas if it only has neighbors in one of $A',B'$, it has at most $n/4 + \zeta n$ by definition of $X$. Combining these estimates gives $u$'s degree is at most $n/4 + 2 \zeta n$.

Since $G_1-X$ is bipartite, it follows that
\begin{align*}
|E(G_1)| &\le \frac{(n-x)^2}{4} + x\cdot \left(\frac{n}{4}+2\zeta n\right)= \frac{n^2}{4}+x\left(-\frac{n}{4} + \frac{x}{4}+2\zeta n\right) < \frac{n^2}{4}+x\left(-\frac{n}{4} + 3\zeta n\right)
\end{align*}
\end{proof}

\begin{claim}\label{clmiii}
 Provided $\zeta$ is suitably small compared to $\vep$,  $|E(\cH_2)| \le x\cdot 2t\vep n$
\end{claim}

\begin{proof}
Every edge of $\cH_2$ has at least one vertex in $X$: if we choose some $u \in X$ and some $v \in A'' \cup B'' \cup X$ adjacent to $u$ in $G_2$, then the edge $uv$ is contained in at most $3t-2$ edges of $\cH_2$.  This gives us a bound of $x(3t-2)\zeta n$ on the number of edges of $\cH_2$ with at least two vertices in $A'' \cup B'' \cup X$.

To count the remaining hyperedges of $\cH_2$, for each $u \in X$ let $\cH_2(u)$ denote the family of edges of $\cH_2$ containing $u$, a vertex in $A'$, and a vertex in $B'$. We then let $A(u)$ denote the set of vertices in $A'$ and $B(u)$ denote the set of vertices in $B'$ that are in a hyperedge in $\cH_2(u)$.  Finally, let $u_1,\dots, u_x$ be the vertices of $X$, and color every edge $vw$ between $A' \cup B'$ with the subset of $\{1,\ldots,x\}$ for which $vwu_i \in \cH$ (either $\cH_1$ or $\cH_2$). 

Now assume that for $v\in A(u_i)$ there exists $w_1,w_2,\dots, w_{t+1}\in B(u_i)$ such that $vw_j$'s color is not $\{i\}$.
We claim that there is a $Tr(B_t)$ with $u_i$ and $v$ as vertices in the smaller part. For the edge $u_iv$ we use an arbitrary hyperedge containing $u_i$ and $v$, without loss of generality the third vertex is not among $w_1,\dots, w_{t}$. For the edge $vw_j$, we use a hyperedge containing $v$ and $w_j$ and a vertex in $X$ other than $u_i$, which exists since $vw_j$'s color includes an index other than $i$. Finally, since each edge $u_iw_j$ lies in an edge of $\cH_2$ (by definition of $B(u_i)$), $u_iw_j$ is heavy and is therefore contained in an edge of $\cH$ that does not include $v$, and the third vertex of that edge cannot be in $B'$ (thus cannot be $w_\ell$) as $B'$ is independent in $G$.

By construction, $|\cH_2(u_i)|$ is bounded above by the number of heavy edges in the induced subgraph $G[A(u_i)\cup B(u_i)]$, and the argument in the preceding paragraph and symmetry show the heavy edges form a subgraph with maximum degree at most $t$.  We now partition $X$ into sets $L$ and $S$ as follows: let $S$ contain all vertices $u_i$ for which $\max\{|A(u_i)|,|B(u_i)|\} < \vep n$, and let $L = X\setminus S$.

If $u_i \in S$ then $G[A(u_i)\cup B(u_i)]$ contains at most $t\cdot \min\{|A(u_i)|,|B(u_i)|\} \le t\cdot \vep n$ heavy edges, yielding that $|\bigcup_{u \in S} \cH_2(u)| \le |S|\cdot t \vep n$.  Similarly, if $u \in L$ then $G[A(u_i)\cup B(u_i)]$ contains at most $tn/2$ heavy edges, so $|\bigcup_{u \in L} \cH_2(u)| \le |L|\cdot tn/2$.  The last step in our argument is to show that $|L| \le \vep x$: this is trivial if $L = \varnothing$ so we assume otherwise.

Let $u_i \in L$: each vertex $v \in A(u_i)$ is adjacent to at least $|B(u_i)|-t-\zeta n \ge \vep n - t - \zeta n > \vep n/2$ vertices of $B(u_i)$ with edges of color $\{i\}$, as $\zeta < \vep/2$.  Thus, the subgraph $G[A(u_i)\cup B(u_i)]$ contains at least $|A(u_i)|\cdot \vep n/2 \ge \vep^2 n^2/2$ edges with color $\{i\}$, which all lie inside $G_1$. 
Since $|E(G_1)| \le n^2/4$, it follows that $|L| \le (2\vep^2)^{-1}$.

Fix $u_i \in L$ and $u_j \in X$, $j \ne i$: let $vw$ be a heavy edge in $G[A(u_i)\cup B(u_i)]$ whose label includes $j$, where we assume $v \in A'$ and $w \in B'$. If $vwu_j \in \cH_2$ then $vw$ is an edge in $G[A(u_j)\cup B(u_j)]$. All such edges must lie in the intersection $G_{i,j}=G[(A(u_i)\cap A(u_j) \cup (B(u_i)\cap B(u_j))]$.  The earlier arguments show the maximum degree in $G_{i,j}$ is at most $2t$ (else some vertex lies on $t+1$ edges with a label other than $i$ or other than $j$). The definition of $A'$ and $B'$ implies that $G_{i,j}$ has at most $2t + \zeta n$ vertices in each part, and therefore there are fewer than $(2t + \zeta n)(2t) < 4t\zeta n$ such heavy edges.

If $vwu_j \notin \cH_2$, then it lies in $\cH_1$ and one of $vu_j,wu_j$ lies in $G_1$. But Claim~\ref{clmi} implies $u_j$ has less than $\zeta n+3t-2< 2\zeta n$ neighbors in $A'$ and in $B'$, since it has neighbors in both parts in $G$. Since the neighbor in $G_1$ determines the heavy edge, we have at most $4\zeta n$ such heavy edges.

Thus, $j$ occurs in the color of at most $4(t+1)\zeta n$ heavy edges in $G[A(u_i)\cup B(u_j)]$, and therefore there must be at least $|A(u_j)|/(4(t+1)\zeta n) \ge \vep/(4(t+1)\zeta)$ indices $j \ne i$. In other words, $x-1\ge\vep/(4(t+1)\zeta)$. Thus, provided $\zeta \le \vep^4/(2(t+1))$, we have
\[|L| \le \frac{1}{2\vep^2} \le \vep\cdot \frac{\vep}{4(t+1)\zeta} \le \vep x.\]

Combining the estimates, it follows that
\[|E(\cH_2)| \le x(3t-2)\zeta n + |S|t\vep n + |L|tn/2 \le x\vep n + x t \vep n + \vep x tn/2 \le x\cdot 2t\vep n.\]
\end{proof}


Provided $\vep$ and $\zeta$ are small enough, Claims~\ref{clmii} and \ref{clmiii} imply
\[|E(\cH)| \le n^2/4 + x(-n/4 + 3\zeta n +  2t\vep n) < (n-1)^2/4-\Omega(n)\]
if $x>2$. We suppose now that $x=2$ and that $|E(\cH)| \ge  (n-1)^2/4-\vep n$ and aim to produce a contradiction.  Since $(n-1)^2/4 = (n-2)^2/4 + (2n-3)/4$ and $\vep$ is small, by Claim~\ref{clmiii} we may assume that $|E(G_1)| >  (n-2)^2/4 + 3n/8$. Since $G-X$ is bipartite, at least $3n/8$ edges of $G_1$ are incident with $X=\{u_1,u_2\}$, so without loss of generality we suppose that $d_{G_1}(u_2)\ge 3n/16$ and that $u_2 \in A$.  Since $|X\cup A'' \cup B''| \le \zeta n$, it follows that $u_2$ has $\Omega(n)$ neighbors in $B'$ in $G_1$.

But then Claim~\ref{clmi} implies $u_2$ has no neighbors in $A'$ in $G$.  Consequently, every edge of $G_1$ in $A' \cup B'$ must form a triple in $\cH$ with $u_1$, implying $u_1$ is adjacent in $G$ to $A' \cup B'$.  This, in turn, implies $u_1$ has fewer than, say, $2\zeta n$ neighbors in $A'$ or in $B'$ in $G_1$ (by Claim~\ref{clmi} again). Thus, $d_{G_1}(u_1) < 3\zeta n$, so $u_2$ has at least
\[d_{G_1}(u_2) -\zeta n\ge \frac{3n}{8}-4\zeta n > \frac{n}{4} + \zeta n \]
neighbors in $B'$. But then $u_2 \in A''$ by definition rather than $X$, the desired contradiction.

Therefore, if $|E(\cH)| \ge  (n-1)^2/4-\vep n$ then we must have $x=1$, implying $\cH=\cH_1 \le \lfloor (n-1)^2/4\rfloor$, and that equality only holds when $\cH$ is formed by adding a new vertex to every edge of a maximum bipartite graph on $n-1$ vertices.
\end{proof}

\bigskip

\textbf{Funding}: Research supported by the National Research, Development and Innovation Office - NKFIH under the grants FK 132060 and KKP-133819.

\end{document}